\pdfoutput=1
\documentclass[10pt,a4paper]{amsart}

\usepackage[margin=16mm]{geometry}

\usepackage{epsfig}
\usepackage{amssymb,amsmath,amsthm,ytableau}
\usepackage{amscd,pb-diagram} 
\usepackage[all]{xypic}
\usepackage[numbers]{natbib}
\usepackage[colorlinks]{hyperref}
\usepackage{MnSymbol}

\def\cprime{$'$} 
  \def\polhk#1{\setbox0=\hbox{#1}{\ooalign{\hidewidth
  \lower1.5ex\hbox{`}\hidewidth\crcr\unhbox0}}}

\newtheorem{theorem}{Theorem}[section]

\newtheorem{corollary}{Corollary}[section]

\newtheorem{example}{Example}[section]

\newtheorem{lemma}{Lemma}[section]

\newtheorem{proposition}{Proposition}[section]

\theoremstyle{definition}
\newtheorem{definition}{Definition}[section]
\newtheorem{remark}{Remark}[section]

\newcommand{\x}{\boldsymbol{x}}
\newcommand{\pp}{\boldsymbol{p}}

\newcommand{\R}{\mathbb{R}}
\newcommand{\C}{\mathbb{C}}

\newcommand{\p}{\mathbb{P}}

\newcommand{\AAA}{\mathbb{A}}

\newcommand{\g}{\mathfrak g}
\newcommand{\gh}{\mathfrak h}

\newcommand{\gl}{\mathfrak{gl}}

\newcommand{\so}{\mathfrak{so}}

\newcommand{\sll}{\mathfrak{sl}}

\newcommand{\Span}[1]{\left\langle#1\right\rangle}

\DeclareMathOperator{\tr}{tr}

\DeclareMathOperator{\Lie}{Lie}
\DeclareMathOperator{\SO}{\mathsf{SO}}
\DeclareMathOperator{\OOO}{\mathsf{O}}
\DeclareMathOperator{\CO}{\mathsf{CO}}

\DeclareMathOperator{\GL}{\mathsf{GL}}
\DeclareMathOperator{\PGL}{\mathsf{PGL}}
\DeclareMathOperator{\SE}{\mathsf{SE}}
\DeclareMathOperator{\Eu}{\mathsf{E}}
\DeclareMathOperator{\SL}{\mathsf{SL}}

\DeclareMathOperator{\Gr}{Gr}

\DeclareMathOperator{\Aut}{Aut}

\DeclareMathOperator{\Aff}{\mathsf{Aff}}

\DeclareMathOperator{\Hess}{hess}

\DeclareMathOperator{\Stab}{Stab}

\newcommand{\E}{\mathcal{E}}

\newcommand{\CC}{\mathcal{C}}
\newcommand{\nuovoW}{\mathcal{W}}
\newcommand{\TT}{\mathcal{T}}

\newcommand{\NN}{\mathcal{N}}

\newcommand{\RR}{\mathcal{R}}

\newcommand{\SSSS}{\mathcal{S}}
\newcommand{\Th}{^\textrm{th}}

\newcommand{\Rd}{^\textrm{rd}}
\newcommand{\Nd}{^\textrm{nd}}

  \newcommand{\ttt}{\boldsymbol{t}}
    \newcommand{\w}{\boldsymbol{w}}

\renewcommand{\nuovoW}{\mathcal{W}}

\begin{document}

\title[Projectively and affinely invariant PDEs]{Projectively and affinely invariant PDEs on hypersurfaces 
}

 \author{Dmitri~Alekseevsky}
   \address{Institute for Information Transmission Problems, B. Karetny
per. 19, 127051, Moscow (Russia) and University of Hradec Kralove, Rokitanskeho 62,
Hradec Kralove 50003 (Czech Republic).}
\email{dalekseevsky@iitp.ru}

 \author{Gianni Manno}
   \address{Dipartimento di Matematica ``G. L. Lagrange'', Politecnico di Torino, Corso Duca degli Abruzzi, 24, 10129 Torino, Italy.}
    \email{giovanni.manno@polito.it}
 \author{Giovanni Moreno}
 \address{Department of Mathematical Methods in Physics,
 Faculty of Physics, University of Warsaw,
ul. Pasteura 5, 02-093 Warszawa, Poland}
 \email{giovanni.moreno@fuw.edu.pl}
\date{\today}
\maketitle

\begin{abstract}
In \cite{alekseevsky2020general} the authors have developed a method for constructing $G$--invariant PDEs imposed on hypersurfaces  of an $(n+1)$--dimensional homogeneous space $G/H$, under mild assumptions on the Lie group $G$. In the present paper the method is applied to the case when $G=\PGL(n+1)$ (resp., $G=\Aff(n+1)$) and the homogeneous space $G/H$ is the $(n+1)$--dimensional projective  $\p^{n+1}$ (resp., affine $\AAA^{n+1}$) space, respectively. The paper's main result is that  projectively or affinely invariant PDEs with $n$ independent and one unknown variables are in one--to--one correspondence with invariant hypersurfaces of the space of \textit{trace--free cubic forms} in $n$ variables with respect the group $\CO(d,n-d)$ of conformal transformations of $\R^{d,n-d}$. Local descriptions are also provided.
\end{abstract}

\medskip


\setcounter{tocdepth}{1}
\tableofcontents

\section*{Introduction}
In this paper we go on constructing $G$--invariant PDEs in one unknown variable defined on $(n+1)$--dimensional $G$--homogeneous manifolds, following the general theoretical scheme developed by the authors in   \cite{alekseevsky2020general}: there the cases when $G$ is either the Euclidean  $\SE(n+1)$, or the conformal group $\CO(n+1)$, were  treated: here, we will deal with the cases when $G$ is either the projective $\PGL(n+1)$, or the affine group $\Aff(n+1)$.\par
%
%
The reason why we treat these last two cases together in a separate paper is that, unlike the two before, they give rise to \emph{third--order} invariant PDEs; in particular, this casts an important bridge with the differential geometry of affine hypersurfaces and, in particular, with the Fubini--Pick invariant. On this concern, see, e.g., \cite[Section 2.2]{An-Min:2015aa}, \cite[Section 1]{doi:10.1002/cpa.3160390606}, \cite[Section 3.5]{MR2003610},   \cite{3rdOderAffPDE_preprint}, as well as the original work of Blaschke \cite{wilhelmblaschke1923}.\par
%
%
%
The vanishing of this invariant defines a $G$--invariant $3\Rd$ order PDE that can be constructed,  according to the general scheme developed in \cite{alekseevsky2020general}, by a suitable choice of a \textit{fiducial hypersurface} of order 3. In view of the tight relationship between the affine and the projective case (see also \cite{MR2406036}), we will state a  result concerning both   in Section \ref{sec.gianni}; technical computations concerning the projective case, that is when  $\mathbb{P}^{n+1}$ is regarded as homogeneous space of $\SL(n+2)$, will be carried out in   Section \ref{secCasPROJ}. Analogous computations for   the affine case, that is when  $\mathbb{A}^{n+1}$ is regarded as homogeneous space of the affine group $\Aff(n+1)$, which  can be   thought of  as a ``restriction'' of the projective case,   will be carried out   in Section \ref{secCasAff}.\par
%


The main result  is   Theorem  \ref{corCasoAff}: each projectively or affinely invariant PDE imposed on hypersurfaces of the $(n+1)$--dimensional  projective or affine space is uniquely given by a $\CO(d,n-d)$--invariant hypersurface of the space of \textit{trace--free cubic forms} in $n$ variables, were symbol $\CO(d,n-d)$ denotes the Lie group of conformal transformations of the space $\R^n$, equipped with a metric of signature $(d,n-d)$.\par 
A local coordinate description for the $\Aff(n+1)$--case, which obviously works for the $\PGL(n+1)$--case as well,   will be given  in Section \ref{secCoordDescr}, whereas in the last Section \ref{subEsempioA3} we focus on the case $n=2$.



 
\subsection*{Notations and conventions} The symmetric product will be denoted by $\odot$, and the symmetric $\ell$--power of a vector space $V$ will be denoted by $S^\ell(V) $. If $f:M\to N$ is a differentiable map, then pull--back via $f$ of a bundle $\pi:E\to N$, is denoted by $f^*(E)$. Symbol $\R^\times$ denotes the multiplicative group of real numbers.

\subsection*{Acknowledgments}
G.~Manno gratefully acknowledges support by the project ``Connessioni
proiettive, equazioni di Monge-Amp\`ere e sistemi integrabili'' (INdAM),
``MIUR grant Dipartimenti di Eccellenza 2018-2022 \linebreak (\texttt{E11G18000350001})'', ``Finanziamento alla Ricerca'' \texttt{53\_RBA17MANGIO},  and
PRIN project 2017 ``Real and Complex Manifolds: Topology, Geometry and
holomorphic dynamics'' (code \texttt{2017JZ2SW5}).  G.~Manno is a member of\linebreak GNSAGA
of INdAM.  G.~Moreno is supported by the Polish National Science Centre grant under the contract number \texttt{2016/22/M/ST1/00542}, as well as by the National Science Center  project “Complex contact
manifolds and geometry of secants”, \texttt{2017/26/E/ST1/00231}. The authors thank the anonymous referee for keen comments and useful suggestions.

\section{A  general construction  of  $G$--invariant  PDEs     on  a  homogeneous  manifold  $M = G/H$}\label{secRipArtPrec}

We will review here, without proofs, the main definitions and results, as well as all the necessary preliminary material, contained in \cite[Sections~2 and~3]{alekseevsky2020general}. Throughout this section  $M=G/H$ will be an $(n+1)$--dimensional homogeneous manifold
and  $S \subset M$ an embedded  hypersurface of $M$; in   Section \ref{secCasPROJ} and  Section \ref{secCasAff}, $M$ will be either   the projective space $\p^{n+1}$, or the affine space $\AAA^{n+1}$, respectively. 
 
 \subsection{Preliminary definitions}\label{sec:preliminaries}

Locally,  in an appropriate  local chart
\begin{equation}\label{eqn:coord.u.x}
(u,\x)=(u,x^1, \ldots, x^n)
\end{equation}
of $M$, the hypersurface $S$  can  be  described  by an  equation   $u = f(\x) =  f(x^1, \dots, x^n)$, where  $f$~is  a smooth function    of the variables $x^1, \dots, x^n$, that we refer to as the \emph{independent} variables, to   distinguish them  from the remaining   coordinate $u$,  that is the \emph{dependent} one.\footnote{A reader 
who is familiar with the standard literature about jet spaces may have noticed that we reversed the order of $\x$ and $u$: this choice will be more convenient for us as the coordinate $u$ will play the role of the ``$0^{\textrm{th}}$ coordinate''.
}
 We say that such  a chart is  \emph{admissible} for $S$ or, equivalently, that the hypersurface $S$ is (locally) admissible for the chart $(u,\x)$.
We denote by $S_f=S$ the graph of $f$:
\begin{equation*}
S_f:=\{\big(f(\x)\,,\x\big)\}=\{u=f(\x)\}\, .
\end{equation*}
%

Given two hypersurfaces $S_1$ and $S_2$ through a common point $\pp $, one can always choose a chart $(u,\x)$ about $\pp$ that is admissible for both: $S_1=S_{f_1}$, $S_2=S_{f_2}$. 

\begin{definition}\label{defDmitriGetti}
Two hypersurfaces $S_{f_1}, \, S_{f_2}$ passing through a common point $\pp = (u,\x)$ are called \emph{$\ell$--equivalent} at $\pp$  if  the Taylor expansions of $f_1$ and $f_2$, in a chart admissible for both, coincide at $\x$ up to order $\ell$.  The  class  of $\ell$--equivalent hypersurfaces to a given hypersurface $S$ at the point $\pp$ is denoted by $[S]^{\ell}_{\pp}$, and the union
\begin{equation*}
J^{\ell}(n,M):=\bigcup_{\pp\in M}\{[S]^{\ell}_{\pp}\mid \text{$S$ is a hypersurface of $M$ passing through $\pp$}\}
\end{equation*}
of all these equivalence classes is the \emph{space of $\ell$--jets of hypersurfaces}   of $M$.
\end{definition}
Note that $J^1(n,M)=\p T^*M$, that is the Grassmanian bundle $\Gr_n(TM)$ of tangent $n$--planes to the $(n+1)$--dimensional manifold $M$. From now on, when there is no risk of confusion, we let
$$
J^{\ell}:=J^{\ell}(n,M)\,.
$$
The natural projections
\begin{equation*}
\pi_{\ell,m}:J^{\ell}\stackrel{}{\longrightarrow} J^{m}\,,\quad [S]^{\ell}_{\pp}\longmapsto [S]^m_{\pp}\, ,\quad \ell>m\, ,
\end{equation*}
define a tower of bundles
$$
\dots\longrightarrow\ J^{\ell}\longrightarrow J^{\ell-1}\longrightarrow\dots\longrightarrow J^1=\mathbb{P}T^*M\longrightarrow J^0=M\,.
$$
It is well known that $\pi_{\ell,\ell-1}$ are affine bundle for $\ell\geq 2$.
For any $a^m\in J^{m}$,  the fiber of $\pi_{\ell,m}$ over $a^m$ will be denoted  by the symbol
$$
J^{\ell}_{a^m}:=\pi_{\ell,m}^{-1}(a^m)\,.
$$
\begin{definition}
A  \emph{system}  of  $m$ PDEs of order $k$ is  an $m$--codimensional submanifold   $\E  \subset  J^k$. A  \emph{solution}  of    the  system   $\E$  is  a  hypersurface   $S \subset M$    such  that $S^{(k)} \subset \E$.
 \end{definition}
\subsection{Assumptions on the Lie group $G$}
Before introducing the conditions   the Lie group $G$ will have to fulfill (see    Section~\ref{sec:descriptions} below) in order to make  Theorem~\ref{thMAIN1} work, we recall some  basic facts about the affine group that will help understand the meaning of these conditions.
\subsubsection{Affine groups and their subgroups of affine type}
Let $V$ be a vector space, treated as an affine space: then the  group $\Aff(V)$ of    affine transformations of $V$  fits into the short exact sequence of groups:
\begin{equation}
  0\longrightarrow  V\stackrel{T}{\longrightarrow}\Aff(V)\stackrel{L}{\longrightarrow}\GL(V)
   \longrightarrow 0\label{eqShExSeq}
\end{equation}
The monomorphism $T$ maps a vector $v\in V$ into the corresponding parallel translation $T_v$: one has then a canonical   normal subgroup $T_V$, made of parallel   translations, which acts  on  $V$ in a simply transitive way.\par
The action of $T_V$   defines  even an  absolute parallelism on $V$, i.e., it allows to  canonically identify the tangent  space  $T_vV$ at an arbitrary point $v\in V$ of the \emph{affine} space $V$, with the \emph{vector}  space  $V$: in particular, if   an origin $o\in V$ is chosen, then the differential 
\begin{equation*}
    L(g):=d_og:T_oV\longrightarrow T_{g\cdot o}V 
\end{equation*}
of  $g\in\Aff(V)$ at $o$ can be regarded as an isomorphism of $V$, that is as an element of $\GL(V)$. This explains the rightmost arrow of \eqref{eqShExSeq} and allows to   regard $\GL(V)$ as the \emph{linear group of the affine group}, that is, as the subgroup   $\Aff(V)_o=\GL(T_oV)$ of the group $\Aff(V)$ that stabilizes the origin  $o\in V$: this leads to the 
   semidirect  decomposition 
   \begin{equation}\label{eqDecomposizioneDmitresca}
       \Aff(V) = T_V \rtimes    \GL(V) 
   \end{equation}
   of the  affine group $\Aff(V)$, associated with the origin $o\in V$.\par
   If now a subgroup   $H \subset\Aff(V)$ is given, decomposition \eqref{eqDecomposizioneDmitresca} does not need to descend to $H$, in the sense that the sequence
\begin{equation}
  0\longrightarrow  T_W:=T^{-1}(H)\stackrel{T}{\longrightarrow}H\stackrel{L}{\longrightarrow}L_H:=L(H)
   \longrightarrow 0\label{eqShExSeqRESTR}
\end{equation}
   may be still exact, but not split.  
   This remark motivates the following definition.
   \begin{definition}\label{defSubAffType}
  We say that a subgroup $H \subset\Aff(V)$ is  of   \emph{affine type}   if $H$  admits  a semidirect decomposition
\begin{equation}\label{eqSubAffType}
    H = T_W \rtimes L_H
\end{equation}
   for  some  $o \in V$. The subgroup $L_H=H_o$ is called the \emph{linear subgroup} of $H$, whereas $T_W$ is its \emph{subgroup of translations}.
   \end{definition}
In condition (A2) below we shall require that  $\tau(H^{(k-1)})$ be a subgroup of affine type; 
indeed, as a direct consequence of Definition~\ref{defSubAffType},  if $H$ is subgroup of affine type of $\Aff(V)$, then   the    orbit $H\cdot o $ of $o$  coincides with the  affine subspace $W$ of $V$.\par 
     Let $H =   T_W  \rtimes L_H $ be  a  subgroup of  affine type, where $ L_H=H_o$  denotes its   linear   subgroup, and let us fix a  complementaty subspace $U$  to $W$ in $V$: then any  $h \in  H$  can be  decomposed into a product 
\begin{equation}\label{eqn:h.T.wh}
h =T_{ h}\cdot   L_h  \, .
\end{equation}
Therefore, in terms  of the  decomposition $V = U + W$, 
    the  action  of the linear  part $L_h$  takes the form
\begin{equation}\label{eqn:Lh}
L_h=
\left(
\begin{array}{cc}
* & *
\\
0 & \overline{L}_h
\end{array}
\right)\,.
\end{equation}
We let $\overline{L}_H:=\{\overline{L}_h\mid h\in H\}$.
\begin{lemma}\label{lemma.aff.banale}
Let $ H \subset \Aff(V)$ be a subgroup of  affine type. Then there exists a 1–1 correspondence between\linebreak  $\overline{L}_H$--invariant
hypersurfaces $ \overline{\Sigma} \subset  \overline{V} = V /W $ and (cylindrical) $H$–invariant hypersurfaces  ${\Sigma} = W + \overline{\Sigma}$ in $V$.
\end{lemma}

\begin{proof}
Let $\pi : V \to \overline{V} = V/W$ be the projection. Then  if $\overline{\Sigma} $ is  an $L_{H}$--invariant hypersureface in $\overline{V}$, then  $\Sigma := \pi^{-1}(\bar{ \Sigma})$ is  an $H$--invariant hypersurface in $V$, see also \cite[Lemma~3.1]{alekseevsky2020general}.
\end{proof}

\subsubsection{$k$--admissible homogeneous manifolds}\label{sec:descriptions}
In what follows, unless otherwise specified, $o$ is a fixed point of  $M=G/H$ (an ``origin''), so that $M=G\cdot o$, and $o^{\ell}$ is a point of $J^\ell$ projecting onto $o$.  This allows us to consider, $\forall\,\ell\geq 2$, the fibre $J^{\ell}_{o^{\ell-1}}$ as   a vector  space   with  the  origin  $o^{\ell}$ playing the role of zero vector.
The  group  $G$    acts naturally on each  $\ell$--jet space $J^{\ell}$:
\begin{eqnarray*}
   g :    J^{\ell}   &\longrightarrow&  J^{\ell}\, ,\\
   o^{\ell}= [S]^{\ell}_{o}  &\longrightarrow& g\cdot o^\ell:=[g(S)]^{\ell}_{g(o)}\,  , 
\end{eqnarray*}
with $o\in S$, for all   $g\in G$.
\begin{definition}
  The system $\E$ is called   \emph{$G$--invariant} if  $G \cdot\E = \E$.
\end{definition}
We denote by $H^{(\ell)}$ the stability subgroup $G_{o^{\ell}}$ in $G$ of the point $o^{\ell}$:
$$
H^{(\ell)}:=G_{o^{\ell}}\,.
$$
%
We are going to   assume that there exists a point $o^k\in J^k$, with $k\geq 2$, such that:\par\medskip
\noindent \textbf{(A1)} the orbit
\begin{equation*}
\check{J}^{k-1} := G\cdot o^{k-1} = G/H^{(k-1)} \subset  J^{k-1}
\end{equation*}
through the projection $o^{k-1}\in J^{k-1}$ of $o^k$ is open;\par\medskip
\noindent \textbf{(A2)} the orbit
\begin{equation}\label{eqn:Wk.orbit}
W^k:=\tau(H^{(k-1)})\cdot o^{k}\subset J^k_{o^{k-1}}
\end{equation}
of the natural affine action
\begin{equation}\label{eqn:tau.tauk.Dmitri}
\tau :  H^{(k-1)} \to \Aff(J^k_{o^{k-1}})
\end{equation}
in the  fibre  $J^k_{o^{k-1}}$ is an affine subspace and the group $\tau(H^{(k-1)})$ is a subgroup of affine type, i.e., 
\begin{equation}\label{eqn:decomp.tau.H}
\tau(H^{(k-1)})=T_{W^k}\rtimes L_{H^{(k-1)}}\, ,
\end{equation}
where $L_{H^{(k-1)}}$ is the stabilizer of $o^k$, see Definition~\ref{defSubAffType}.\par
Assumption (A2) implies that  there is a point $o^{k} \in J^k_{o^{k-1}}$   such  that   the    restriction    of   the    affine  bundle   $\pi_{k,k-1} : J^k\to J^{k-1}$ to  the orbit   $G\cdot o^{k}$   is  an  affine  subbundle  of  $\pi_{k,k-1}$ (over the base $\check{J}^{k-1}$).

\begin{definition}\label{def.Dmitri}
A  homogeneous  manifold   $M = G/H$  is called \emph{$k$--admissible}  for  $k \geq 2$ if assumptions (A1) and (A2) are satisfied.
\end{definition}

The problem of classifying all  $G$--invariant PDEs  $\E\subset J^k$ on a given $(n+1)$--dimensional manifold $M$ acted upon by a Lie group $G$ will be made more workable  by assuming $M$
to be a $G$--homogeneous manifold of a particular kind, namely a $k$--admissible one.%
%
%
%

\subsection{Natural bundles   on jet spaces}\label{SecMainRes}

\subsubsection{The lift of  hypersurfaces of $M$ to $J^\ell$}\label{subJetsSubs}

The space $J^{\ell}$ has a natural structure of smooth manifold: one way to see this is to extend the   local coordinate system \eqref{eqn:coord.u.x} on $M$
to  a  coordinate system
\begin{equation}\label{eqn:jet.coordinates}
(u,\x,\ldots,u_i,\ldots, u_{ij}, \ldots, u_{i_1 \cdots i_l},\ldots) = (u,x^1,\dots,x^n,\ldots,u_i,\ldots, u_{ij}, \ldots, u_{i_1 \cdots i_l},\ldots)
\end{equation}
on $J^{\ell}$, where each coordinate function\footnote{The $u_{i_1 \cdots i_k}$'s are symmetric in the lower indices.} $u_{i_1 \cdots i_k}$, with $k\leq \ell$,  is unambiguously defined by 
\begin{equation}\label{eqDefHigCoordJets}
u_{i_1 \cdots i_k}\left([S_f]^{\ell}_{\pp}\right)= \partial^k_{i_1 \cdots i_k}f(\x)\, ,\quad \pp=(u,\x)\,,\quad k\leq \ell\,.
\end{equation}
In formula \eqref{eqDefHigCoordJets} above the symbol $\partial_i$ denotes the partial derivative $\partial_{x^i}$, for  $i= 1,  \ldots, n$; we recall that the hypersurface $S=S_f$ is the graph of the function $u=f(\x)$ and, as such, it is   admissible for the chart $(u,\x)$.\par
%
The $\ell$--lift of $S$ is defined by
\begin{equation*}
S^{(\ell)}:=\{[S]^{\ell}_{\pp}\,\,|\,\,\pp\in S\}\,.
\end{equation*}
It is an $n$--dimensional submanifold of $J^{\ell}$. If $S=S_f$ is the graph of  $u=f(\x)$, then $S_f^{(\ell)}$ can be naturally parametrized as follows:\footnote{We stress once again that a switch has occurred between the first and the second entry, with respect to a more standard literature.}
\begin{equation*}
\left(u=f(\x),\x,\dots u_i=\frac{\partial f}{\partial x^i}(\x) ,\dots  u_{ij}=\frac{\partial^2 f}{\partial x^i\partial x^j}(\x) , \dots \right)\,.
\end{equation*}
%

\subsubsection{The tautological bundle and the higher order contact distribution on $J^{\ell}$}
\begin{lemma}\label{lemma.Dmitri}
Any point $a^{\ell}=[S]^{\ell}_{\pp}\in J^{\ell}$ canonically defines the $n$--dimensional subspace
\begin{equation}\label{eq:Dmitri.1}
T_{a^{\ell-1}}S^{(\ell-1)}\subset T_{a^{\ell-1}}J^{\ell-1}\,,\quad a^{\ell-1}:=\pi_{\ell,\ell-1}(a^{\ell})\,.
\end{equation}
\end{lemma}
\begin{definition}\label{defTautBundle}
The \emph{tautological} rank--$n$ vector bundle $\TT^{\ell} \subset \pi_{\ell,\ell-1}^* (TJ^{\ell-1})$   is the bundle over $J^{\ell}$ whose fiber over the point $a^{\ell}$ is given by  \eqref{eq:Dmitri.1}, i.e.,
\begin{equation*}
\TT^{\ell}=\left\{ (a^{\ell},v)\in J^{\ell}\times TJ^{\ell-1}\,\,|\,\,v\in  T_{a^{\ell-1}}S^{(\ell-1)}\right\} \,.
\end{equation*}
\end{definition}
The (truncated) total derivatives
\begin{equation}\label{eqn:total.derivatives}
D_i^{(\ell)}:=\partial_{x^i}+\sum_{k=1}^{\ell}\sum_{j_1\leq\dots\leq j_{k-1}}u_{j_1\dots j_{k-1}\,i}\,\partial_{u_{j_1\dots j_{k-1}}}\,,\quad i=1\dots n\, ,
\end{equation}
constitute a local basis of the bundle $\TT^{\ell}$.\par

%
The pre--image $\CC^{\ell}:=(d\pi_{\ell,\ell-1})^{-1}\TT^{\ell}$ of the tautological bundle on $J^{\ell}$, via the differential $d\pi_{\ell,\ell-1}$ of the canonical projection $\pi_{\ell,\ell-1}$, is a distribution on $J^{\ell}$.
\begin{definition}
 $\CC^{\ell}$ is called the \emph{$\ell\Th$ order contact structure} or the  \emph{Cartan distribution} (on $J^{\ell}$).
\end{definition}
We will need also the \emph{vertical subbundle} $T^vJ^{\ell}:=\ker (d\pi_{\ell,\ell-1})$ of $TJ^{\ell}$. The distribution 
%
$\mathcal{C}^{\ell}$ has been called the ``higher order contact structure'' \cite{KrasilshchikLychaginVinogradov:GJSNPDEq,MR2352610,MR722524}  because, for $\ell=1$,   if $(u,x^i,u_i)$ is a chart on $J^1$, then $\mathcal{C}:=\mathcal{C}^1=\ker(\theta)$, where  
$d\theta=du-u_idx^i$,  
is the  contact distribution.  

\subsubsection{The affine structure of the  bundles   $J^{\ell}\to J^{\ell-1}$ for $\ell\geq 2$}\label{secSezioneDedicataADmitri}

According to   Definition \ref{defTautBundle}, the tautological bundle $\TT:=\TT^1$ is the vector bundle over $J^1$ defined by
\begin{equation*}
\TT_{[S]_{\pp}^1}:=\TT^1_{[S]_{\pp}^1}=T_{\pp}S\,  .
\end{equation*}
\begin{definition}
The \emph{normal bundle} $\NN$ is the  line bundle 
\begin{equation*}
\NN_{[S]_{\pp}^1}:=N_{\pp}S= T_{\pp}M\big/T_{\pp}S\,
\end{equation*}
over   $J^1$. 
\end{definition}
\begin{remark}\label{remarkCheEraPiuLungo}
To simplify notations, we denote by $\partial_u$ the equivalence class $\partial_u\,\mathrm{mod}\,\TT$.
\end{remark}

Lemma \ref{lemma.ker.vector} and Proposition \ref{prop:affine.bundles} below are both   well known (see for instance \cite{KrasilshchikLychaginVinogradov:GJSNPDEq,MR989588}).
%
\begin{lemma}\label{lemma.ker.vector}
For $\ell\geq 1$, the following vector bundle isomorphism holds: 
\begin{equation*}
T^vJ^{\ell} \simeq \pi_{\ell,1}^* (S^{\ell}\TT^* \otimes \NN)\, .
\end{equation*}
\end{lemma}
%
%
\begin{proposition}\label{prop:affine.bundles}
For $\ell\geq 2$, the bundles $J^{\ell}\to J^{\ell-1}$ are affine bundles modeled by the vector bundles $\pi_{\ell-1,1}^* (S^{\ell}\TT^* \otimes \NN)$. In particular, once  a  chart $(u,\x)$ has been fixed, a choice of a point $[S]^{\ell}_{\pp}$ (the origin) defines an identification of  $J^{\ell}_{[S]^{\ell-1}_{\pp}}$ with $S^{\ell}T^*_{\pp} S$.
\end{proposition}


\subsection{Constructing $G$--invariant   PDEs $\E$} \label{secConstructingGinvPDEs}

Let $M = G/H  =  G\cdot o$, $o\in M$, be  an  $(n+1)$--dimensional  homogeneous manifold and recall  (see Section \ref{sec:descriptions}) that $G$ acts on each  jet space $J^\ell=J^{\ell}(n,M)$.  
%
%
%
To further simplify the setting, we will assume that $M=G/H$ possess a \emph{fiducial hypersurface} of order $k$, defined below.
\subsubsection{The fiducial hypersurface} 
\begin{definition}\label{defFidHyp}
Let  $S \subset  M$ be a hypersurface,  such that $S\ni o$. The hypersurface $S$ is called a a \emph{fiducial hypersurface} (of order $k$), if $S$  is homogeneous  with respect to a  subgroup  of $G$, such that (A1) and (A2) of Section \ref{sec:descriptions} are satisfied with   $o^k:=[S]^k_o$.
\end{definition}
Plainly, if $M=G/H$ admits  a {fiducial hypersurface} of order $k$, then it is $k$--admissible as well (see Definition \ref{def.Dmitri}).  
Let $S$ be a fiducial hypersurface of order $k$ in the sense of Definition \ref{defFidHyp}:  therefore, for any $\ell\leq k$, we will regard the point 
$$
o^{\ell}:= [S]^{\ell}_o \in J^{\ell}
$$
  as the origin of $J^{\ell}$. Furthermore,  the identification
$$
J^{\ell}_{o^{\ell-1}}= S^{\ell}(T_{o}^*S)\otimes N_oS\, ,
$$
in the case when the  fiducial hypersurface $S$ is the graph $S_f$ of a $f$,   reads
 (see   Proposition \ref{prop:affine.bundles}):
\begin{equation}\label{eqn:identification}
J^{\ell}_{o^{\ell-1}}= S^{\ell}(T_{o}^*S_f)\,.
\end{equation}
We will  use  this identification in the sequel.

\subsubsection{A general method for constructing $G$--invariant PDEs}
We   apply now Lemma \ref{lemma.aff.banale} to the   subgroup $\tau(H^{(k-1)})\subset \Aff(J^k_{o^{k-1}})$ of affine type, which eventually leads to  \cite[Theorem~3.1]{alekseevsky2020general}.
\begin{corollary}\label{cor.1.1.corresp}
Let $M=G/H$ be a $k$--admissible homogeneous manifold.
Then there exists a $1$--$1$ correspondence between $L_{H^{(k-1)}}$--invariant hypersurfaces  $\overline{\Sigma}\subset J^k_{o^{k-1}}/W^k$ and (cylindrical) $\tau(H^{(k-1)})$--invariant hypersurfaces $\Sigma=p^{-1}(\overline{\Sigma})\subset J^k_{o^{k-1}}$,  where
\begin{equation}\label{eqn:p.proj}
p: J^k_{o^{k-1}}\to J^k_{o^{k-1}}/W^k
\end{equation}
is the natural projection.
\end{corollary}
The aforementioned main result of \cite{alekseevsky2020general}, that is  Theorem~3.1, is a direct consequence of above Corollary \ref{cor.1.1.corresp} and below Lemma \ref{lemLemmaLemmino}, applied to the bundle $\pi_{k,k-1}$.

\begin{lemma}\label{lemLemmaLemmino}
Let  $\pi : P \longrightarrow B$  be  a  bundle. Assume that a Lie group $G$ of automorphisms  of  $\pi$,  such  that $B =  G/H$,  acts  transitively on  $B$,  where $H$ is  the  stabilizer of a point  $o \in B$.  Then:
\begin{itemize}
\item [i)] any $H$--invariant function $F$ on $P_o:=\pi^{-1}(o)$ extends to a $G$--invariant  function  $\widehat{F}$ on  $P$ (where  $\widehat{F}(gy) =F(y)$ for  $y \in  P_o$   and     $g \in  G$), and   $F\longmapsto\widehat{F}$ is a  bijection;
\item[ii)] any $H$--invariant   hypersurface   $\Sigma$  of  the  fiber  $P_o$ extends to  a $G$--invariant   hypersurface   $\mathcal{E}_{\Sigma} := G \cdot\Sigma$  of  $P$, and this   gives a bijection between $H$--invariant hypersurfaces of $P_o$ and $G$--invariant hypersurfaces of $P$.
\end{itemize}
\end{lemma}
\begin{proof}
See \cite[Lemma~3.2]{alekseevsky2020general}.
\end{proof}
%
%
%
\begin{theorem}\label{thMAIN1}
Let $M=G/H$ be a $k$--admissible homogeneous manifold (see Definition \ref{def.Dmitri}). Then there is  a natural  $1$--$1$   correspondence  between   $L_{H^{(k-1)}}$-- invariant hypersurfaces $\overline{\Sigma}$  (see also \eqref{eqn:decomp.tau.H}) of $J^k_{o^{k-1}}/W^k$
and  $G$--invariant  hypersurfaces  $\E_{\overline \Sigma}:= \E_{p^{-1}(\overline{\Sigma})}=   G \cdot p^{-1}(\overline{\Sigma})$  of  $J^k=J^k(n,M)$, where $p$ is the natural projection \eqref{eqn:p.proj}.
\end{theorem}
%
Above Theorem \ref{thMAIN1} closes the summary of the  theory developed by the authors in  \cite{alekseevsky2020general}, that is a    strategy  for    constructing     $G$--invariant  PDEs   imposed on  the hypersurfaces  of   a  $k$--admissible homogeneous manifold $M = G/H$:
%
%

\begin{enumerate}
\item    calculate   the orbit   $W^k = \tau(H^{(k-1)})\cdot o^{k} $   and decompose $\tau(H^{(k-1)})$ accordingly  to \eqref{eqn:decomp.tau.H};
\item   describe $L_{H^{(k-1)}}$--invariant  hypersurfaces $\overline{\Sigma} \subset {V}^k=J^{k}_{o^{k-1}}/W^k$;
\item   write  down the  $G$--invariant   equations   $\E_{\overline \Sigma}  =  G \cdot p^{-1}(\overline{\Sigma})$ in the  coordinates \eqref{eqn:jet.coordinates}.
\end{enumerate}

In  the   next  Section \ref{secCasPROJ}  we begin  implementing   this strategy  for the  projective space $\p^{n+1}$, whereas in Section \ref{secCasAff} we will be dealing with the affine space $\AAA^{n+1}$; the $G$--invariant PDE itself is obtained, in an unified manner, in the Section \ref{sec.gianni}.



\section{Stabilizers of the $\SL(n+2)$--action on    $J^\ell(n,\p^{n+1})$}\label{secCasPROJ}

We consider the linear space  $\nuovoW:=\R^{n+2}$  with the basis     
$$\{p, e_1,\ldots, e_n,q \}$$ 
and we  let $G=\SL(n+2)$ act naturally on it; therefore, $G$ acts on the projectivization $M:=\p \nuovoW$   of $\nuovoW$. The projective coordinates
\begin{equation*}
[u:x^1:\cdots: x^n:t]
\end{equation*}
on $\p \nuovoW=\p^{n+1}$ will be given by the dual  coordinates to the basis above.
We shall also need a
scalar product
\begin{equation}\label{eqScalProdErrEnnPROJ}
g=\langle \,\cdot\, , \,\cdot\,  \rangle
\end{equation}
on $ E:=\Span{e_1,\ldots,e_n}$, of  signature  $d,n-d$.
Let $\SSSS_g$ denote 
the  \emph{projective quadric}
\begin{equation}\label{eqSupFidPROJ}
\SSSS_g:=\p \nuovoW_0 \, ,
\end{equation}
where $\nuovoW_0$ is the null cone of the pseudo-euclidean metric  
\begin{equation}\label{eqn:gW}
 g_{\nuovoW}:=g-du\odot dt\, ,
\end{equation}
that is, $\nuovoW_0:=\{w\in\nuovoW\,|\,g_{\nuovoW}(w,w)=0\}$. 
%
In Section~\ref{sec.gianni}  we shall  prove that $\SSSS_g$ is a fiducial hypersurface, see Proposition~\ref{corExPropFidHypCasoAff}.\par 
The point 
$$
o:=[p]=[1:0:\cdots: 0:0]
$$ 
clearly belongs to the   hypersurface $\SSSS_g$, so that it makes sense to consider  
$$o^{(k)}:=[\SSSS_g]_o^k$$ for $k\geq 0\,.$ 
In particular,  the point 
$o^{(1)}=[\SSSS_g]^1_o$, 
that is the tangent space $T_o\SSSS_g=T_o(\SSSS_g\cap\mathcal{U})\in J^1$, in the affine coordinate neighborhood 
\begin{equation}\label{eqn:affine.neigh}
\mathcal{U}:=\{ [1:x^1:\dots:x^n:t] \}
\end{equation}
can be identified with $E=\ker d_ot$: indeed,  $\ker d_o( tu-g)=\ker (d_ot- d_og)=\ker d_ot=E$, because $d_og=0$.

\begin{lemma}\label{lemStrutturaSottogruppiCasoPROJ}
  The stabilising subgroups corresponding to the origins $o^{(k)}$, for $k=0,1,2,3$, are:
\begin{eqnarray*}
H &=&\R^{n+1}\rtimes\GL(n+1)\, ,
\\
H^{(1)} &=& \R^{n+1}\rtimes((\R^n\rtimes\GL(n))\times\R^\times)\, ,
\\
H^{(2)} &=& (\R^{n+1}\rtimes(\R^n\rtimes\OOO(d,n-d))\times\R^\times) \, ,
\\
H^{(3)} &=&  \R^{n+1}\rtimes(\OOO(d,n-d) \times\R^\times) \, .
\end{eqnarray*}
\end{lemma}
\begin{proof}
An element of  $G$ stabilizing  the line generated by $p$ is a $(n+2)\times(n+2)$  matrix with determinant one, displaying all zeros in the first column, save for the first entry, that has to be equal to the inverse of the determinant of the rightmost lower  $(n+1)\times(n+1)$ block: in other words,
 \begin{equation}\label{eqPrimaFormaAcca}
H=G_{[p]}=\Aff(E\oplus \R q )=\R^{n+1}\rtimes\GL(n+1)\, .
\end{equation}
The same can be seen on the infinitesimal level: passing to the Lie algebra $\g$ of $G$, we consider the decomposition
\begin{equation}\label{eqDecSLW}
\g=\sll(\nuovoW)=\so(\nuovoW)\oplus S^2_0(\nuovoW) \, ,
\end{equation}
where $\so(\nuovoW)=\so(d+1,n+1-d)$ is identified with the space of skew--symmetric forms $\Lambda^2 \nuovoW$ and $S^2_0(\nuovoW)$ denotes the space of trace--free symmetric forms with respect to $g_{\nuovoW}$, cf. \eqref{eqn:gW}; therefore, since   $\nuovoW$ splits into the sum
\begin{equation*}
\nuovoW=\R p\oplus ( E\oplus\R q)
\end{equation*}
of the  $(n+1)$--dimensional space $\R q\oplus E$ and the one--dimensional subspace $\R p$, we obtain   the decompositions
\begin{eqnarray}
\so(\nuovoW)&=&(\R p\wedge    ( E\oplus  \R q))\oplus\so( E\oplus \R q ) \, ,\label{eqDecSOW1}\\
S^2_0(\nuovoW)&= & \R p \odot  (E  \oplus \R p ) \oplus \R q \odot  (E  \oplus \R q ) \oplus S^2_0E \oplus \R(p \odot q - e_0 \otimes e_0) .\label{eqDecSOW2}
\end{eqnarray}
where $e_0 \in E$ is a suitable vector.
%
%
%
%
%
%
%
It is now easy to see that   the Lie algebra $\g_0=\gh$ of the stabilizer $H$ is given by 
\begin{equation}\label{eqAccaPiccoloProj}
\gh=\gl\,(E\oplus \R q) \,\,\,\oplus\,\,\, (\R p\otimes(E\oplus \R q)^*) =
\gl\,(E\oplus \R q) \,\,\,\oplus\,\,\, ((\R p)^*\otimes(E\oplus \R q))^*\, .
\end{equation}



Since
\begin{equation*}
T_oM=T_{[p]}\p \nuovoW=(\R p)^*\otimes\frac{\nuovoW}{\R p}\simeq (\R p)^*\otimes (   E\oplus \R q)\, ,
\end{equation*}
we obtain 
\begin{equation*}
\gh\simeq T_oM\oplus\gl(T_oM)\, .
\end{equation*}
%
%
The last identification allows to rewrite \eqref{eqPrimaFormaAcca} as follows:
 \begin{equation*}
H= T_oM\rtimes\GL(T_oM)\, ,
\end{equation*}
where the factor $\GL(T_oM)$ (resp., $T_oM$) is the image (resp., kernel) of the isotropy representation
\begin{equation}\label{eqIsotropyProj}
j:H\longrightarrow \GL(T_oM)\, .
\end{equation}
In light of what we have found it is easy to pass from \eqref{eqAccaPiccoloProj} to the Lie algebra $\gh^{(1)}$ of $H^{(1)}$
\begin{equation}\label{eqAccaUnoProj}
\gh^{(1)}=\gh_{o^{(1)}}=(\R p\otimes (E\oplus  \R q )^*)\oplus \gl(E\oplus  \R q)_{E}\, ,
\end{equation}
where
\begin{equation}\label{eqAccaUnoProjBIS}
\gl(E\oplus  \R q)_{E}= \gl(E)\oplus E\oplus  \R q  \, ,
\end{equation}
is the subalgebra preserving $E$.
On the level of Lie groups this means that
\begin{equation*}
H^{(1)}= (   E\oplus\R q)^*\rtimes ((E\rtimes\GL(E))\times (\R q)^\times)\simeq \R^{n+1}\rtimes ((\R^n\rtimes\GL(n))\times\R^\times)\, ,
\end{equation*}
or, more intrinsically,
\begin{equation}\label{eqAccaUnoProjGRUPPO}
H^{(1)}=T_oM\rtimes(\Aff(T_o\SSSS_g)\times\R^\times)\, .
\end{equation}
We shall show now that the subgroup of \eqref{eqAccaUnoProjGRUPPO} that stabilizes $o^{(2)}$ is precisely
\begin{equation}\label{eqAccaDueProj}
H^{(2)}=T_oM\rtimes(\Eu(T_o\SSSS_g)\times\R^\times)\, ,
\end{equation}
where
\begin{equation*}
\Eu(T_o\SSSS_g)=E\rtimes\OOO(E)=\R^n\rtimes\OOO(d,n-d)
\end{equation*}
is the group of rigid motions of $E\simeq T_o\SSSS_g$.
Note that  the isotropy representation   \eqref{eqIsotropyProj}  tells us that the factor $ T_oM$ of $H^{(1)}$ survives in $H^{(2)}$; it is also easy to see that the ``conformal factor'' $\R^\times$, since it scales the dependent variable, does not affect the second jet at zero of the quadric $\SSSS_g$, see \eqref{eqSupFidPROJ}: indeed, if we identify second--order jets with quadratic forms (see Section \ref{secSezioneDedicataADmitri}), then  the second jet at zero of the quadric $\SSSS_g$ is $g$ itself. Similarly, a transformation coming from the $\GL(T_o\SSSS_g)$ component of  the group $\Eu(T_o\SSSS_g)$ preserves $o^{(2)}$ if and only if it preserves $g$: therefore, it must be an element of $\OOO(T_o\SSSS_g)$.\par
In order to finish the proof of \eqref{eqAccaDueProj} it   remains to show that the ``translational'' component $T_o\SSSS_g$ of $\Eu(T_o\SSSS_g)$ does not move $o^{(2)}$: 
we postpone this  to the proof of the analogous property in Section~\ref{secCasAff} below (see Remark~\ref{remLem3121}), together with the proof  that the aforementioned component does move $o^{(3)}$, eventually showing that
\begin{equation*}
H^{(3)}=T_oM\rtimes(\OOO(T_o\SSSS_g)\times\R^\times)\, ,
\end{equation*}
thus concluding the whole proof.
\end{proof}
\begin{remark}
The structure of $H^{(1)}$ is that of
\begin{equation*}
H^{(1)}=\mathcal{H}\rtimes\Aut(\mathcal{H})\, ,
\end{equation*}
where
\begin{equation*}
\Lie(\mathcal{H})=E\oplus E^*\oplus \R p\, ,
\end{equation*}
with  $\mathcal{H}$ being the  $(2n+1)$--dimensional Heisenberg group. Indeed, from
\eqref{eqAccaUnoProj} and \eqref{eqAccaUnoProjBIS} it follows that
\begin{equation*}
\gh^{(1)}=\gh_{o^{(1)}}=(\gl(E)\oplus E\oplus  \R q ) \oplus (\R p\otimes E)\oplus\R p\, ,
\end{equation*}
that is,
\begin{equation}\label{eqAccaUnoHeisenberg}
\gh^{(1)}= \Lie(\mathcal{H})   \oplus\gl(E)\oplus\R q\, .
\end{equation}
In terms of traceless $(n+2)\times(n+2)$ matrices, an element 
\begin{equation*}
(  a, b^t, \alpha,A,\beta )
\end{equation*}
of the algebra \eqref{eqAccaUnoHeisenberg} corresponds to the matrix
\begin{equation*}
\left(\begin{array}{ccc}- \beta -\tr (A) & 0 & \alpha \\ 0 & A &  b \\0 & a^t &\beta \end{array}\right)\, .
\end{equation*}
\end{remark}
\begin{remark}
The   hypersurface  $\SSSS_g$ is a homogeneous manifold: namely,
\begin{equation*}
\SSSS_g=\SO(\nuovoW)/(\SO(\nuovoW)\cap H^{(1)})\, .
\end{equation*}
\end{remark}
It is indeed convenient, before passing to the application of Theorem \ref{thMAIN1}, to prove the analogous result for the affine case: after that, the two cases will go on in parallel, due to the fact that the structure of the model fiber of $J^3$ over $J^2$ does not feel the topology of the underlying manifold, that has changed from $\p^{n+1}$ to $\mathbb{A}^{n+1}$.

\section{Stabilizers of the $\Aff(n+1)$--action on    $J^\ell(n,\AAA^{n+1})$}\label{secCasAff}

By the symbol $\mathbb{A}^{n+1}$ we denote the linear space $\R^{n+1}$, regarded as an affine space. The affine space $\mathbb{A}^{n+1}$ is a manifold with the action of the affine group
\begin{equation*}
G=\Aff(n+1)=\R^{n+1}\rtimes \GL(n+1)\, ,
\end{equation*}
such that the vector normal subgroup $\R^{n+1}$ acts simply transitively.  We fix the standard basis
$$
\{e_0,e_1,\ldots, e_n \}\, 
$$ 
of $\R^{n+1}=\R e_0\oplus E$, and  we let 
$$(u,\x):=(u,x^1,\ldots, x^n)$$
be the corresponding    coordinates.
We  have then  the same $n$--dimensional space $E$ as before, with the same coordinates, but now the $(n+1)$--dimensional underlying manifold is
\begin{equation*}
M:=\mathbb{A}^{n+1}=\R^{n+1}=\R e_0\oplus E\, .
\end{equation*}
Since $\mathbb{A}^{n+1}$ still possesses the zero, we set   $o:=0\in\mathbb{A}^{n+1}$. 
In analogy to~\eqref{eqSupFidPROJ} we let $\SSSS_g^{\textrm{aff}}$ be 
the \emph{quadric}
\begin{equation}\label{eqSupFidAFF}
\SSSS_g^{\textrm{aff}}=\{ u=g( \x,\x)\} \, ,
\end{equation}
where $g$ is the same
scalar product on $E$ as before, see~\eqref{eqScalProdErrEnnPROJ}. 
In Section~\ref{sec.gianni}  we shall  prove that $\SSSS_g^{\textrm{aff}}$ is a fiducial hypersurface, see Proposition~\ref{corExPropFidHypCasoAff}.\par  
As before, we let $o^{(k)}:=[\SSSS_g^{\textrm{aff}}]_o^k$, for $k\geq 0$, so that  the point $o^{(1)}=[\SSSS_g^{\textrm{aff}}]_o^1$ will be again the tangent space $T_o\SSSS_g^{\textrm{aff}}\in J^1$, that is the hyperplane  $E=\R^n$ of $\R^{n+1}$.

\begin{lemma}\label{lemStrutturaSottogruppiCasoAffine}
The stabilizing subgroups of the origins $o^{(k)}$, for $k=0,1,2,3$, are:
\begin{eqnarray}
H &=& \GL(n+1)\, ,\label{eqH0casoAff}\\
H^{(1)} &=&  (\R^n\rtimes\GL(n))\times\R^\times\, ,\label{eqH1casoAff}\\
H^{(2)} &=&  (\R^n\rtimes\OOO(d,n-d))\times\R^\times\, ,\label{eqH2casoAff}\\
H^{(3)} &=&  \OOO(d,n-d)\times\R^\times\, .\nonumber
\end{eqnarray}
\end{lemma}
\begin{proof}
Formula \eqref{eqH0casoAff} is well--known: an affine transformation preserves the zero, that is the origin $o^{(0)}=o=0$ of $\AAA^{n+1}$, if and only if it is linear, i.e., an element of $\GL(n+1)$.\par
Concerning  \eqref{eqH1casoAff}, let us note  that a $(n+1)\times (n+1)$ non--singular matrix preserves the hyperplane $\R^n$, that is the origin $o^{(1)}$, if and only if it has the form
\begin{equation*}
\left(\begin{array}{cc}A & \w \\0 & \mu\end{array}\right)\, ,
\end{equation*}
where $A\in\GL(n)$, $\w\in\R^n$ and $\mu\in\R^\times$. Identity 
 \begin{equation*}
\left(\begin{array}{cc}A & \w \\0 & \mu\end{array}\right)=\mu\cdot \left(\begin{array}{cc}\mu^{-1}A & \mu^{-1} \w \\0 & 1\end{array}\right) \ 
\end{equation*}
shows   that $\Stab_H(o^{(1)})$ is obtained from  the subgroup $\R^n\rtimes\GL(n)$ of matrices of the form
\begin{equation}\label{eqMatrAffEnnePErEnner}
\left(\begin{array}{cc} A &  \w \\0 & 1\end{array}\right)\, ,
\end{equation}
by multiplying it by the group $\R^\times$.\par  

To deal with~\eqref{eqH2casoAff}, it is convenient to introduce, by a slight abuse of notation, two special elements of $H^{(1)}$, namely
\begin{equation*}
M_{A,\mu} := \left(\begin{array}{cc}A & 0 \\0 &\mu\end{array}\right)\, , \quad
\w := \left(\begin{array}{cc}I_n & \w \\ 0 &1\end{array}\right)\, .
\end{equation*}
It is worth observing that $M_{A,\mu}$ acts naturally by $A$ on the hyperplane $\R^n$, while rescaling by $\mu$ the elements of the   complementary line $\R e_0$, whereas the vector $\w$ acts on the affine hyperplane $u=1$ by translation: in particular, it ``tilts'' the  line $\R e_0$ into the line   $\R(e_0+\w)$.\par 
%
Since
\begin{equation*}
\left(\begin{array}{cc}A & \w \\0 & \mu\end{array}\right)=\left(\begin{array}{cc}A & 0 \\0 & \mu\end{array}\right)\cdot \left(\begin{array}{cc}I_n & A^{-1}\w \\0 & 1\end{array}\right)\, ,
\end{equation*}
 any element of $H^{(1)}$ can be expressed  as  product of the above special elements:
\begin{equation}\label{eqFormulaAggiuntaDalReferee}
   H^{(1)}=\{ M_{A,\mu}\cdot \w\mid  A\in\GL(n)\, ,\mu\in\R^\times\, ,  \w\in\R^n\} .
\end{equation}
Let us pass to the first claim of \eqref{eqH2casoAff}, i.e., to the computation of the stabilizer $ \Stab_{H^{(1)} }(o^{(2)})$ of the second--order jet at $0\in\R^n$ of the quadric hypersurface $\SSSS_g^{\textrm{aff}}=\{ ( Q(\x),\x)\mid \x\in\R^n\}$, where  
$$
Q(\x):=g(\x,\x)
$$
is the quadratic form associated to the scalar product \eqref{eqScalProdErrEnnPROJ}.\par

To begin with, we act by a transformation of type $\w$    on the   hypersurface $\SSSS_g^{\textrm{aff}}$: it turns out that, even if the resulting hypersurface $\w(\SSSS_g^{\textrm{aff}})$   looks like a ``slanted paraboloid" (see the picture below), the second--order jet at zero of $\w(\SSSS_g^{\textrm{aff}})$ is the same as the original hypersurface  $\SSSS_g^{\textrm{aff}}$.
\par\bigskip\noindent\centerline{\epsfig{file=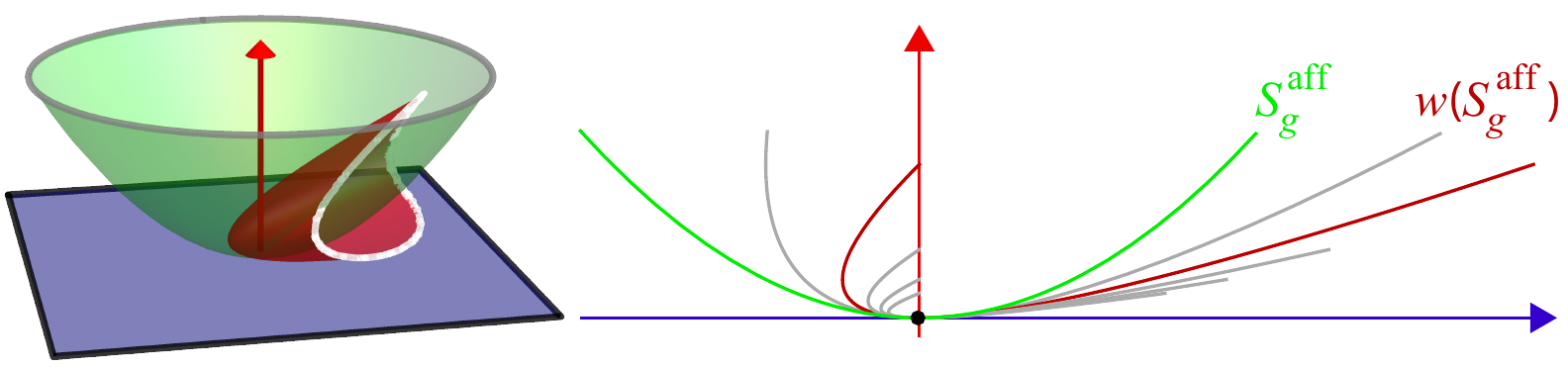,width=\textwidth}}\par\bigskip
In order to see this,  let us observe that 
\begin{equation*}
\w(\SSSS_g^{\textrm{aff}})=\{ \w\cdot (Q(\x),\x)\mid \x\in\R^n\}= \{ Q(\x), (\x+ Q(\x)\w)\mid \x\in\R^n\} \, ,
\end{equation*}
where  the function $\ttt(\x):=\x+ Q(\x)\w  $ is a small deformation of the identity in a   sufficiently small neighborhood of zero. As such, $\ttt(\x)$ will admit a (local) inverse. We claim, that
\begin{equation}\label{eqPseudoInversa}
\x(\ttt):=\ttt-Q(\ttt)\w\,
\end{equation}
approximates   the inverse of $\ttt(\x)$ up to third--order terms. Indeed,
\begin{align*}
\ttt(\x(\ttt))&=\ttt-Q(\ttt)\w+Q(\ttt-Q(\ttt)\w)\w\nonumber\\
&=\ttt-Q(\ttt)\w+(Q(\ttt)-2Q(\ttt)\langle \w,\ttt\rangle+Q^2(\ttt)Q(\w))\w\nonumber\\
&= \ttt- (2Q(\ttt)\langle \w,\ttt\rangle+Q^2(\ttt)Q(\w))\w\nonumber\\
&=\ttt +O(\|\ttt\|^3)\, .
\end{align*}
This will allow us to work with the graph of the function $f(\ttt):=Q(\x(\ttt))$ instead of the hypersurface $\w(\SSSS_g^{\textrm{aff}})$, as long as only  jets from  zero up to second order are concerned. In particular,
\begin{equation*}
\nabla f (0)= \nabla Q (0)\cdot \frac{\partial \x}{\partial \ttt}(0)=\nabla Q (0)\, ,
\end{equation*}
since the Jacobian $ \frac{\partial \x}{\partial \ttt}$ at zero is the identity. We have then proved that $[f]_0^1=o^{(1)}$. Analogously,
\begin{equation*}
\frac{\partial^2 f}{\partial t^i\partial t^j}(0)
=
\frac{\partial}{\partial t^j}\left( \frac{\partial Q }{\partial  x^k} \frac{\partial  x^k}{\partial  t^i}\right)(0)
 =
\frac{\partial^2 Q  }{\partial x^h\partial x^k } \frac{\partial x^h  }{\partial t^j }  \frac{\partial x^k }{\partial t^i }(0)+ \frac{\partial Q }{\partial  x^k}  \frac{\partial^2 x^k  }{\partial t^j \partial t^i }(0)
  =  \frac{\partial^2 Q  }{\partial x^h\partial x^k }(0)\delta^h_j\delta^k_i
   =  \frac{\partial^2 Q  }{\partial x^i\partial x^j } (0)\,,
\end{equation*}
since the first derivatives of $Q$  vanish at the origin. Then we also have that $[f]_0^2=o^{(2)}$, i.e.,   $\w\cdot o^{(2)}=o^{(2)}$.\par 
Therefore, in view of~\eqref{eqFormulaAggiuntaDalReferee},  
\begin{equation}\label{eqIntermedia}
   H^{(2)}=\{ M_{A,\mu}\in H^{(1)}  \mid  M_{A,\mu}\cdot o^{(2)}=o^{(2)}\} ,
\end{equation}
so that it remains to compute the second--order jet at zero of the hypersurface
\begin{equation*}
M_{A,\mu}(\SSSS_g^{\textrm{aff}})=\{   (\mu Q(\x),A\cdot\x)\mid \x\in\R^n\}= \{ (\mu A^{-1\,\ast}(Q)(\x),\x), \mid \x\in\R^n\} \, ,
\end{equation*}
and impose that it be equal to $o^{(2)}$.\par 
%
%
Since both $Q$ and $\mu A^{-1\,\ast}(Q)$ are quadratic forms, their second--order jets at zero coincides if and only if 
\begin{equation*}
\mu A^{-1\,\ast}(Q)=Q\Leftrightarrow A^*(Q)=\mu Q\, , 
\end{equation*}
i.e., $A$ is a conformal transformation of $Q$, and $\mu$ is the corresponding conformal factor, uniquely determined by $A$. In other words,
  $A\in \OOO(d,n-d)\cdot\R^\times$ and $\mu^n=\det(A)^2$, 
which concludes the proof that $H^{(2)} =(\R^n\rtimes\OOO(d,n-d))\times\R^\times$.\par
The last case, i.e., the second claim of \eqref{eqH2casoAff}, will be dealt with in a similar fashion: to compute $H^{(3)}$, we first rewrite $H^{(2)}$ as 
%
\begin{equation*}
   H^{(2)}=\{ M_{ A,\mu}\cdot \w\mid  A\in\OOO(d,n-d)\times\R^\times\, ,\  \mu^n=\det(A)^2\, ,\   \w\in\R^n\} ,
\end{equation*}
in analogy to~\eqref{eqFormulaAggiuntaDalReferee}.\par 
Since both $Q$ and $\mu A^{-1\,\ast}(Q)$ have vanishing third--order jets at zero, in order to compute $H^{(3)}$ it suffices to impose that transformations of type $\w$ preserve $o^{(3)}$, i.e.,
\begin{equation*}
   H^{(3)}=\{ M_{ A,\mu}\cdot \w\in H^{(2)}  \mid  \w\cdot o^{(3)}=o^{(3)}\} ,
\end{equation*}
in analogy to the previous case~\eqref{eqIntermedia}. In this last case, however,  
%
  the third--order jet at zero of $\w(\SSSS_g^{\textrm{aff}})$ will not be the same as $o^{(3)}$, unless  $\w=0$.  We have already observed that $f(\ttt)$ and $Q(\x)$ have the same derivatives at 0 up to order two.\par
To study the third--order jet at zero of $\w(\SSSS_g^{\textrm{aff}})$  we need    compute   the third derivatives of $f$, where now $f(\ttt)=Q(\x(\ttt))$, with $\x(\ttt)$ being the \emph{true inverse} of $\ttt(\x)$, and not the approximated one, i.e.,
\eqref{eqPseudoInversa}.  The reason why we use the same symbol for both the exact and the approximated (local) inverse, beside an evident     notation simplification, is that the final  result will depend  only on the approximated one.
\begin{align*}
 \frac{\partial^3 f  }{\partial t^i\partial t^j\partial t^l }&=  \frac{\partial  }{\partial t^j }\left(  \frac{\partial^2 Q  }{\partial x^h\partial x^k } \frac{\partial x^h  }{\partial t^j }  \frac{\partial x^k }{\partial t^i } + \frac{\partial Q }{\partial  x^k}  \frac{\partial^2 x^k  }{\partial t^j \partial t^i } \right)\nonumber\\
 &= \frac{\partial^3 Q}{\partial x^h\partial x^k \partial x^s}\frac{\partial x^s }{\partial t^l}\frac{\partial x^h }{\partial t^j}\frac{\partial  x^k}{\partial t^i}+\frac{\partial^2 Q  }{\partial x^h\partial x^k}\frac{\partial^2 x^h  }{\partial t^j\partial t^l}\frac{\partial x^k  }{\partial t^i}+\frac{\partial^2 Q  }{\partial x^h\partial x^k}\frac{\partial^2 x^k  }{\partial t^i\partial t^l}\frac{\partial x^h  }{\partial t^j}+\frac{\partial^2 Q  }{\partial x^k\partial x^s}\frac{\partial^2 x^k  }{\partial t^j\partial t^i}\frac{\partial x^s  }{\partial t^l}+\frac{\partial Q}{\partial x^k}\frac{\partial^3 x^k}{\partial t^i\partial t^j\partial t^l  }\, .
\end{align*}
Evaluating the last expression at 0 we obtain
\begin{align}
 \frac{\partial^3 f  }{\partial t^i\partial t^j\partial t^l }(0)&=  2Q_{hk}\frac{\partial^2 x^h  }{\partial t^j\partial t^l}(0)\delta_i^k+2Q_{hk}\frac{\partial^2 x^k  }{\partial t^i\partial t^l}(0)\delta_j^h+2Q_{ks}\frac{\partial^2 x^k  }{\partial t^j\partial t^i}(0)\delta_l^s\nonumber\\
 &=2Q_{hi}\frac{\partial^2 x^h  }{\partial t^j\partial t^l}(0) +2Q_{jk}\frac{\partial^2 x^k  }{\partial t^i\partial t^l}(0) +2Q_{kl}\frac{\partial^2 x^k  }{\partial t^j\partial t^i}(0)\, .\label{eqDerTerzaEffeInZero}
 \end{align}
 Now, for the purpose of computing the second derivatives of $\x$ at 0 in \eqref{eqDerTerzaEffeInZero}, we can use the approximated inverse, that is
 \eqref{eqPseudoInversa}:
\begin{equation*}
\frac{\partial^2 \x(\ttt)  }{\partial t^i\partial t^j}(0)=\frac{\partial^2 (\ttt-Q(\ttt)\w)  }{\partial t^i\partial t^j}(0)=-2Q_{ij}\w\, .
\end{equation*}
Indeed, the discrepancy between the true and the approximated inverse, being of third order in $\x$, will still vanish in 0, even after a double   differentiation. \par
Therefore,   the third--order term of the   Taylor expansion of $f$ around 0 (where, it is worth stressing, $f$ is the one computed via the true inverse of $\ttt(\x)$) is  precisely
   \begin{equation}\label{eqFormaModificata}
\frac{1}{3!} \frac{\partial^3 f  }{\partial t^i\partial t^j\partial t^l }(0)t^it^jt^l=-\frac{1}{6}\left(2Q_{hi}2Q_{jl}w^h +2Q_{jk}2Q_{il}w^k +2Q_{kl}2Q_{ji}w^k \right)t^it^jt^l=-2Q(\ttt)\langle \ttt,\w\rangle\, .
\end{equation}
Since we have already observed that $\w(\SSSS_g^{\textrm{aff}})$ and $\SSSS_g^{\textrm{aff}}$ have the same jets at 0 up to order 2, and the third--order derivatives of  $Q$ are zero, formula \eqref{eqFormaModificata} shows that $[\SSSS_g^{\textrm{aff}}]_0^3=o^{(3)}$   if and only if $\w=0$.\par
This shows that $ \Stab_{H^{(2)} }(o^{(3)})=\OOO(d,n-d)\times\R^\times$, thus concluding the entire proof.
\end{proof}
\begin{remark}\label{remLem3121}
As we have anticipated, the proof of Lemma \ref{lemStrutturaSottogruppiCasoAffine} above also provides the missing steps in the proof of Lemma \ref{lemStrutturaSottogruppiCasoPROJ}; observe also that the residual action of the group $G$ on the fiber $J^3_{o^{2}}$ is exactly the same, that is, that of $\Eu(\R^n)\times \R^\times$. It is then reasonable to continue analyzing the two cases in parallel.
\end{remark}

\section{$\PGL(n)$-- and $\Aff(n)$--invariant PDEs on hypersurfaces of $\p^{n+1} $ and $\AAA^{n+1}$}\label{sec.gianni}


\begin{proposition}\label{corExPropFidHypCasoAff}
The projective hyperquadric $\SSSS_g$ defined by \eqref{eqSupFidPROJ} (resp. the quadric hypersurface $\SSSS_g^{\textrm{aff}}$ defined by \eqref{eqSupFidAFF}) is a fiducial hypersurface of order both 2 and 3 with respect to the action of the affine group $\Aff(n+1)$ on the affine space $\AAA^{n+1}$ (resp., of the projective group $\SL(n+2)$ on the projective space $\p^{n+1}$), in the sense of Definition~\ref{def.Dmitri}.
\end{proposition}

\begin{proof}
For the order $k=2$ the proof is analogous  to the Euclidean case, see     \cite[Proposition~4.1]{alekseevsky2020general}. Indeed,   $J^1$ is  the same as  $\p T^*\R^{n+1}$ or, equivalently, the flag space $F_{0,n}$, on which the linear group $\GL(n+1)$ already acts transitively, let alone $\Aff(n+1)$.  So, $o^{(1)}$ is the flag $(0,\R^n)$ and the action of $H$ on   $J^1_0=\p(\R^{n+1\, \ast})$ is  transitive. Therefore, since  the $\Aff(n+1)$--orbit of $o$ is the entire $M$,   the $\Aff(n+1)$--orbit  of $o^{(1)}$ is the entire space $J^1$, viz.
\begin{equation*}
J^1(n, \AAA^{n+1})=\Aff(n+1)/H^{(1)}\, .
\end{equation*}
To deal with the case $k=3$ we shall  study the orbit $H^{(1)}\cdot o^{(2)}$ in $J^2_{o^{(1)}}$, bearing in mind the identification 
\begin{equation}\label{eqLaSolitaIntramontabileIndentificazione}
J^2_{o^{(1)}}\equiv S^2 T_o^*\SSSS_g^{\textrm{aff}}\otimes N_o\SSSS_g^{\textrm{aff}}=S^2\R^{n\,\ast}\otimes \Span{\partial_u}\, ,
\end{equation}
cf.~\eqref{eqn:identification}, and the description \eqref{eqH1casoAff} of $H^{(1)}$. Since the quadratic form $Q$ associated to the scalar product is non--degenerate, its  $\GL(n)$--orbit will be open. Incidentally, we see the appearance of a $\Aff(n+1)$--invariant \emph{second--order} PDE, namely the Monge--Amp\`ere equation $\E\subset J^2$ given by $\det(u_{ij})=0$.\par
Summing up,
\begin{equation*}
\check{J}^2=\Aff(n+1)\cdot o^{(2)}=\Aff(n+1)/H^{(2)}\,
\end{equation*}
is an open subset of $J^2(n, \AAA^{n+1})$ (which is contained in the complement $ J^2(n, \AAA^{n+1})\smallsetminus\E$
of the Monge--Amp\`ere equation $\E$). Therefore, the assumption (A1) of Definition \ref{def.Dmitri} is met for  the order $k=3$.\par
It remains to check assumption (A2) of Definition \ref{def.Dmitri}: we begin by showing that the orbit $H^{(2)}\cdot o^{(3)}$ is a proper affine sub--space of $J^3_{o^{(2)}}$. To this end, we shall need the identification  
\begin{equation}\label{eqLaSolitaIntramontabileIndentificazioneTRE}
J^3_{o^{(2)}}\equiv  S^3\R^{n\,\ast}\otimes \Span{\partial_u}\, ,
\end{equation}
that is analogous to~\eqref{eqLaSolitaIntramontabileIndentificazione}. 
Indeed, from the proof of Lemma~\ref{lemStrutturaSottogruppiCasoAffine} above it is clear that the $H^{(2)}$--orbit of $o^{(3)}$ is made of the elements
\begin{equation*}
[\w(\SSSS_g^{\textrm{aff}})]_o^3\, ,
\end{equation*}
with $\w\in\R^n$. Therefore, from formula \eqref{eqFormaModificata} it follows immediately that
\begin{equation*}
[\w(\SSSS_g^{\textrm{aff}})]_o^3-o^{(3)}=[ -2Q(\x)\langle \x, \w\rangle ]_0^3\,
\end{equation*}
and then \eqref{eqLaSolitaIntramontabileIndentificazioneTRE} allows to identify the difference $[\w(\SSSS_g^{\textrm{aff}})]_o^3-o^{(3)}$ with the element
\begin{equation}\label{eqElementoDifferenzaCasAff}
-2\w^\#\odot g
\end{equation}
of the vector space $S^3\R^{n\,\ast}\otimes \Span{\partial_u}$, where $\w^\#$ is the dual covector to $\w$ by means of the scalar product \eqref{eqScalProdErrEnnPROJ}.
In other words, as $\w$ ranges in $\R^n$, \eqref{eqElementoDifferenzaCasAff} describes the linear subspace
\begin{equation*}
\R^{n\,\ast}\odot \Span{g}\subset S^3\R^{n\,\ast}\, .
\end{equation*}
By construction, this is the linear space modeling the fibre $H^{(2)}\cdot o^{(3)}$. Since the same is true for any fibre,  assumption (A2) of Definition \ref{def.Dmitri} is met; indeed, as we pointed out in Section \ref{sec:descriptions} above, assumption (A2), in the case when   \eqref{eqn:decomp.tau.H} holds, is the same as having a (proper) affine sub--bundle, and  \eqref{eqn:decomp.tau.H} 
immediately follows from \eqref{eqH2casoAff}.\par
The projective case can be dealt with analogously.
\end{proof}

\subsection{The main result}

\begin{theorem}\label{corCasoAff}
Fix a scalar product $g$ of signature $(d,n-d)$ as in  \eqref{eqScalProdErrEnnPROJ}  and let $\SSSS_g^{\textrm{aff}}\subset\AAA^{n+1}$ (resp., $\SSSS_g\subset\p^{n+1}$) be the corresponding fiducial (quadratic)
hypersurface.
Let
\begin{equation}\label{eqS03.traceless}
S_0^3\R^{n\,\ast}:=\frac{S^3\R^{n\,\ast}}{\R^{n\,\ast}\odot\Span{g}}
\end{equation}
denote the space of \emph{trace--free cubic forms} on $\R^n$. Then, for any $\CO(d,n-d)$--invariant hypersurface
\begin{equation*}
\Sigma\subset S_0^3\R^{n\,\ast}\, ,
\end{equation*}
we obtain an  $\Aff(n+1)$--invariant third--order PDE $\E_\Sigma\subset J^3(n,\AAA^{n+1})$ (resp., an  $\SL(n+2)$--invariant third--order PDE $\E_\Sigma\subset J^3(n,\p^{n+1})$).
\end{theorem}
\begin{proof}
Let us begin with the affine case. The first step consists in proving  that $\tau_\RR(H^{(2)})$--invariant hypersurfaces in
\begin{equation*}
\frac{S^3T_o^* \SSSS_g^{\textrm{aff}}  \otimes N_o\SSSS_g^{\textrm{aff}}}{R_{o^{(2)}}}
\end{equation*}
 are the same as $\CO(p,n-p)$--invariant hypersurfaces in $S_0^3\R^{n\,\ast}$. To this end, recall the structure of $H^{(2)}$, studied in Lemma~\ref{lemStrutturaSottogruppiCasoAffine} (see, in particular, formula~\eqref{eqH2casoAff}) and observe that the factor $\R^\times$ acts by multiplication by $\mu\in\R^\times$ on $N_o\SSSS_g^{\textrm{aff}}$. The factor $\OOO(p,n-p)$ acts naturally on $S^3T_o^*\SSSS_g^{\textrm{aff}}$, which can be identified with $S^3\R^{n\,\ast}$. According to Proposition \ref{corExPropFidHypCasoAff} above,   an element $\w$ in the factor $\R^{n}$ acts by shifting along $R_{o^{(2)}}=\R^{n\,\ast}\odot\Span{g}$ by $-2\w^\#\odot g$, see also   \eqref{eqElementoDifferenzaCasAff}, 
and hence its action on the quotient is trivial.\par
The claim then follows from    Theorem \ref{thMAIN1}, recalling that, up to a covering, $\CO(d,n-d)=\OOO(d,n-d)\times\R^\times$.\par
Since the projective case can be dealt with analogously, we omit the proof.
\end{proof}
 
\subsection{Coordinate description}\label{secCoordDescr}

Since the problem is, by its nature, a local one, we shall not consider the projective case, since the affine space $\AAA^{n+1}$ can be considered as an affine neighborhood embedded in $\p^{n+1}$.
Again, we extend the global  coordinate  system $\{ u,x^1,\ldots, x^n \}$  of   $\AAA^{n+1}$   to a (local) coordinate system of  $J^3(n,\AAA^{n+1})$: see  also  Section~\ref{subJetsSubs}.
\begin{lemma}\label{lemPrimoLemmaSemplificatoreAFF}
 Let   $\E_\Sigma$ be the   $\Aff(n+1)$--invariant equation associated  to the $\CO(d,n-d)$--invariant hypersurface  $\Sigma$, as in Theorem  \ref{corCasoAff} above. Then, in the aforementioned coordinate system on $J^3$, the equation  $\E_\Sigma$ can be described as $\{ f=0\}$, where
the function $f=f(u_{ij},u_{ijk})$, that  does not depend on  $u,x^1,\ldots, x^n,u_1,\ldots,u_n $, is the same function describing the hypersurface $\Sigma_{o^{(1)}}$ of $J^3_{o^{(1)}}$.
\end{lemma}

\begin{proof}
It is a consequence of Lemma \ref{lemLemmaLemmino}, where the bundle  is
\begin{equation*}
J^1\times J^3_{o^{(1)}}\subset J^3(n,\AAA^{n+1})
\end{equation*}
and the subgroup $T\subset G=\Aff(n+1)$ will be the $(2n+1)$--dimensional group
\begin{equation*}
T=\R^{n+1}\rtimes\left\{  \left(\begin{array}{cc}I_n & 0 \\\w & 1\end{array}\right) \mid \w\in\R^n \right\}\, .
\end{equation*}
The first factor of $T$ acts by translations on $\R^{n+1}$ and  the lifted translations fix the $u_i$'s and, similarly, the $u_{ij}$'s  and the $u_{ijk}$'s. Therefore, it is enough  the first factor of $T$ to fulfil  the    hypothesis of Lemma \ref{lemLemmaLemmino}.\par
Let us consider now
\begin{equation*}
\phi=\left(\begin{array}{cc}I_n & 0 \\\w & 1\end{array}\right)\, .
\end{equation*}
Easy computations show that $\phi^{(1)\,\ast}(u_i)=u_i+w_i$, whereas $\phi^{(2)\,\ast}(u_{ij})=u_{ij}$ and $\phi^{(3)\,\ast}(u_{ijk})=u_{ijk}$. The first fact shows that $T$ acts transitively on $J^1$ (since the translations act transitively on $J^0$ and the $\phi$'s act transitively on the fibres of $J^1\to J^0$). The second fact shows that $T$ acts trivially on the fibre $J^3_{o^{(1)}}$. Thus, the result follows from Lemma \ref{lemLemmaLemmino} applied to the  group $T$.
\end{proof}
\begin{example}
 For $n=2$, a straightforward computation based on the proof of Lemma \ref{lemPrimoLemmaSemplificatoreAFF} (see \cite[Section 6]{3rdOderAffPDE_preprint} for more details), shows that the subset  $\E:=\{f=0\}$ of $J^3$, where
 \begin{align}
f&=6 u_{xx} u_{xxx} u_{xy} u_{yy}
   u_{yyy}-6 u_{xx} u_{xxx} u_{xyy}
   u_{yy}^2-18 u_{xx} u_{xxy} u_{xy}
   u_{xyy} u_{yy}+12 u_{xx} u_{xxy}
   u_{xy}^2 u_{yyy} \label{eqFormulaSospirataXXX}\\
  &-6 u_{xx}^2 u_{xxy}
   u_{yy} u_{yyy}
   +9 u_{xx} u_{xxy}^2
   u_{yy}^2-6 u_{xx}^2 u_{xy} u_{xyy}
   u_{yyy}+9 u_{xx}^2 u_{xyy}^2
   u_{yy}+u_{xx}^3 u_{yyy}^2-6 u_{xxx}
   u_{xxy} u_{xy} u_{yy}^2\nonumber\\
 &+12 u_{xxx}
   u_{xy}^2 u_{xyy} u_{yy}-8 u_{xxx}
   u_{xy}^3 u_{yyy}+u_{xxx}^2 u_{yy}^3\, ,\nonumber
\end{align}
is invariant with respect to the group $\Aff(3)$.  In \cite{3rdOderAffPDE_preprint} it is also shown   that the same subset $\E$, in the real case, shows  two different characters, depending on whether it projects over the open subset $\det u_{ij}>0$, or $\det u_{ij}<0$: the former  corresponds to the invariant PDE associated with $\CO(2)=\CO(0,2)$, the latter to the invariant PDE associated with $\CO(1,1)$; see  also Section \ref{subEsempioA3} below. In the first case, the invariant PDE is actually a \emph{system} of two PDEs: this corresponds to \eqref{eqFormulaSospirataXXX} being the sum of two positive quantities; in the second case, the invariant subset $\E$ turns out to be the  \emph{union} of two scalar PDEs. 
\end{example}

\subsection{Complex  $\CO_n$-invariant  hypersurfaces   in $ S^3_0(\C^n)$, with  $n=3,4$}
The departing point of the main Theorem~\ref{corCasoAff}  is a $\CO(d,n-d)$--invariant hypersurface $\Sigma$ in the trace--free third symmetric power $S_0^3\R^{n\,\ast}$ of the $n$--dimensional real vector space $\R^{n\,\ast}$. While a general classification in the real case is still unattainable, much can be said in the case of small values of $n$, if we work  over the field of complex number.\par
Therefore, only in this section, $V = \mathbb{C}^n$  is going to be a complex vector space, with $n=3,4$: having set $W := S^3_0(V)$,   we shall study complex  $\CO(V)$--invariant  hypersurfaces $\Sigma$ in $W$; in particular, there will be no signature, so that we consider the complex conformal group $\CO(V)=\CO_n(\C)$, rather that its split real counterparts $\CO_{d,n-d}(\R)$.\par
More accurately, we will derive a description of   complex invariant  hypersurfaces   $ \Sigma$ in  the  irreducible  $\CO(V)$--module  $W = S^3_0(V)$  of  traceless  symmetric 3--forms  of the    standard  module $V = \mathbb{C}^n $  for $n=3$   and,  partially,  for  $n=4$  from  the  known   results  of invariants' theory, see \cite{Parshin1994-wu, Springer1977-qd}; afterwards, one can reduce    the  description    of the real   hypersurfaces that are invariant  with respect to the 
  corresponding  normal real forms  $\CO_{1,2}(\R)$  and  $\SO_{2,2}(\R)$, as well as with respect to the  compact  real forms  $\CO_3(\R)$ and $\CO_4(\R)$, to the     description of the real forms of the above--obtained complex   hypersurfaces.\par
By employing the same notation of \cite{Onishchik1990-by}, we  will   denote by  $R(k\pi_1)$ the   irreducible  representation
   of the  simple Lie algebra  $\mathfrak{so}_n(\C)$,  whose  highest weight  is $k \pi_1$,  always assuming that   $n \geq 3$  and denoting  by  $\pi_1$    the  first  fundamental weight of $\mathfrak{so}_n(\C)$: in particular,   $R(\pi_1)$ is the tautological  representation  in the  space $V = \C^n$  and  $R(3\pi_1)$ is the highest  irreducible  component $W=S^3_0V$ in  the     symmetric  cube  $S^3V$.
   \subsubsection{The complex case with $n=3$}
    Recall  that  the  Lie  algebra $\mathfrak{so}_3(\C) $   is  isomorphic   to the Lie algebra $\sll(U) = \sll_2(\C)$   and  that all  irreducible $\sll_2(\C)$--modules   are exhausted by   the   symmetric power $S^k U$ of the tautological   module  $U = \C^2$.  The  tensor product  $S^k U \otimes S^{\ell}U'$  is decomposed    into irreducible submodules by the Klebsh--Gordon formula
    $$  S^kU \otimes S^{\ell}U' = \sum_{i=0}^{\infty} S^{k + \ell -2i}U. $$
    The   tautological  representation   of $\so_3(\C) = \sll_2(\C)$   is the  adjoint  representation   $ V = S^2U $ and the  representation
    $$ R(3 \pi_1) = S_0^3(V) = S^3_0(S^2V) = S^6(U).  $$
    This is  the   $\sll_2(\C)$-module  of  binary  forms   of order  $6$. The  full  algebra  of  (polynomial) invariants  $\C[S^6(U)]^{\sll_2(\C)}$ is known,  see \cite{Parshin1994-wu}.  It is   generated  by 5  invariants  $f_2, f_4,f_6,f_{10}, f_{15}$, of degrees 2, 4, 6, 10, 15, where  the   last  invariant  $f_{15} \in A:=\C [f_2, f_4,f_6, f_{10} ]$  and  the  algebra $A$ is the  algebra of polynomials  in  four  (independent)   variables $f_i$.
     
    \begin{theorem}
    Any  complex  $\SO(V)$--invariant hypersurface  in $S^3_0V$, with $V = \C^3$  has the  form $\Sigma^c_f = \{ f= c \}$   where  $f \in  \C[ f_2, f_4, f_6, f_{10} f_{15}]$ and $c\in\C$ is a constant.  Any $\CO(V)$-invariant  hypersurface  has the form
    $\Sigma^0_f = \{ f=0 \}$ where   $f = f(f_2, f_4, f_6, f_{10}, f_{15})$ is  a homogeneous  polynomial of $f_i$, $\deg(f_i)=i$.
        \end{theorem}
 Moreover,  any   homogeneous   invariant   hypersurface   of degree    $\leq  210$ has the  form  $f=0$  where    the  polynomial  $f$ is  given in  Table~\ref{tab1}; the explicit  form  of the  generators    can be found  in \cite{Springer1977-qd}.\par 
    \begin{table}\caption{Invariant hypersurfaces for $n=3$ and $d\leq 10$}\label{tab1}
    \begin{tabular}{l|l}
        degree $d$ &  polynomial $f$ \\
        \hline
        2 & $f=f_2 $\\ 4 & $f=a f_2^2 + b f_4$  \\
        6 & $f=a f_2^2 + b f_2 f_4 + c f_6$\\
        8 &  $f=a f_2^4 + b f_2^2 f_4 + c f_2 f_6 + d f_4^2$ \\                                10 & $f=a f_2^5 + b f_2^3 f_4 + c f_2^2 f_6 + d f_4 f_6 + e f_{10}$\, 
    \end{tabular}
    \end{table}

\subsubsection{The complex case with $n=4$}

    Consider now the case $n=4$. Then $ \so(V) =\so_4(\C) = \so(U) + \so(U'),\,  U = U' = \C^2$ and the  tautological module  is  $V = U \otimes U'$.
     Then  $S^2_0V = S^2U \otimes S^2U'$
 and
 $$V \otimes S^2_0(V) = SU \otimes S^2U \otimes U' \otimes S^2U' = (S^3U + U) \otimes  (S^3U' + U').$$
  Then   $S_0^3 V =  S^3 U \otimes S^3 U' $.\par 
 It is known that  the  algebra  of invariants of the $\sll_2(\C)$-module  $S^3 U$ of  ternary forms   is  generated  by the  discriminant  $\delta$, see \cite{Springer1977-qd}, where for
 $$p(x,y) = a_0 x^3 + a_1 x^2y + a_2 xy^2 + a_3y^3 $$
 the discriminant is 
 $$\delta(p) = a_1^2 a_2^2-4 a_0a_2^3 -4 a_1^3 a_3 -27a_0^2a_3^2+ 18 a_0 a_1 a_2 a_3.$$
  Hence the algebra of invariants $\C[S^3_0 V ]^{\so_3(\C)}$ contains $\delta$ and $\delta'$.
  \begin{theorem}
  Any polynomial  $f= f(\delta, \delta') $  defines  an invariant hypersurface  $f = c$, where $c\in\C$ is constant. 
   Any  homogeneous polynomial  $f= f(\delta, \delta')$ defines  an $\CO(V)$ invariant  hypersurface  $f=0$.
     \end{theorem}
 We stress that not all invariants   are  polynomials of $\delta$ and $\delta'$, that is, there    may be other invariants.
\subsubsection{A glimpse into the real case}
A standard method to cook out real invariants, having at one's disposal the complex ones, is by means of the anti--involution  $\sigma$ in $\so(\C^n)$, i.e.,   the  complex conjugation  in $\C^n$: the  anti--involution  $\sigma$   determines  the  real form   $\so(k, \ell)$   and then    the  real  and the imaginary  parts of the   complex  generators of   the algebra   $A = \C[\C^n]^{\so(\C^n)}$  turn out to be  real  invariants that  generate  the whole  real  algebra  of invariants. But there is a catch: the so--obtained real generators might be   dependent.\par 
Even though, in general, the    description of   a  minimal system of generators   of $A$  is  a very complicated  problem, in practice it is possible to  describe   the   invariants  in  small degrees $k$.\par
The  so--called   symbolic method for  constructing   invariants boils down  to obtaining   scalar  invariants   by contracting   tensor products
   $y_1 \otimes y_2 \otimes   \cdots \otimes y_k$  of   cubic  forms  with the  inverse   metric  $g^{ij}$.   For example,   for  $k=2$, one can    construct  the   invariant
   $I =  y_{ijk} z_{i' j' k'} g^{i i'} g^{j j'} g^{kk'}$;   in  the  case of binary  form,  one has  to use  also the  determinant $\det y_{ij}$.\par 
This way, one  can  get a  description of the   invariants  in  small  degree.

\subsection{The $\Aff(3)$ case}\label{subEsempioA3} Going back to the real--differentiable setting, if we set $n=2$, then it is is easy to use the results contained into Theorem \ref{corCasoAff} and Lemma \ref{lemPrimoLemmaSemplificatoreAFF} above   to  write down explicitly the unique $\Aff(3)$--invariant scalar third--order PDE $\E$ imposed on hypersurfaces of $\AAA^3$. To clarify what we mean by ``unique'', it should be stressed from the outset that, in general,  the $\Aff(n+1)$--invariant PDE   $\E\subset J^3$ constructed according to Theorem \ref{corCasoAff} projects onto an open subset $\check{J}^2$ of $J^2$: this is a direct consequence of the assumption (A1) on the action of $G$, see above Section \ref{sec:descriptions}. In turn, there are as many open subsets  $\check{J}^2$, as the $\GL(n)$--equivalence classes of fiducial hypersurfaces \eqref{eqSupFidAFF}: if we denote by $d_n$   the ceiling of $n/2$, then these classes are labeled by the signatures
\begin{equation*}
    (n,0)\, ,(n-1,1)\, ,\ldots\, , (n-d_n, d_n)\, ,
\end{equation*}
i.e., there is $d_n+1$ of them. The union of all the subsets $\check{J}^2$ is dense in $J^2$ and its boundary is the unique $2\Nd$ order $\Aff(n+1)$--invariant PDE, that is the Monge--Amp\`ere equation $\det\Hess(u)=0$: see also the proof of the assumption (A1) of Proposition \ref{corExPropFidHypCasoAff}.\par
In the case $n=2$, we have only two open subsets of $J^2$, corresponding to the Riemanian  $(+,+)$ and to the Lorenzian $(+,-)$ signature of the Hessian of the surface in $\AAA^3$, denoted respectively by $\check{J}^2_+$ and $\check{J}^2_-$. 
In view of the important link between the $\Aff(3)$--invariant PDEs and the geometry of affine surfaces,
%
below we sketch the relation between such PDEs and the Fubini-Pick invariant.


Let $u=f(x^1,\dots,x^n)$ describe a hypersurface $S$ of $\AAA^{n+1}$ which is the graph of the function $f$. Let us consider the basis
$$
\big(\partial_u\,,\,D^{(1)}_1\,,\dots,D^{(1)}_n\big)=\big(\partial_u\,,\,\partial_{x^1}+u_1\partial_u\,,\dots\,,\partial_{x^n}+u_n\partial_u\big)\,.
$$
The above basis is \emph{unimodular} as $\det(\partial_u\,,\,D^{(1)}_1\,,\dots,D^{(1)}_n)=1$. The components of the Blaschke metric $G$ are
\begin{equation}\label{eqn:Blaschke.metric}
G_{ij}=\rho u_{ij}
\end{equation}
where
\begin{equation*}
\rho=\left[\det\left( u_{ij}\right)\right]^{-\frac{1}{n+2}}\,,
\end{equation*}
whereas the component of the Fubini-Pick cubic form $C$ are
$$
C_{ijk}=-\frac12\left(  \rho u_{ijk} + f_{ij}D_k(\rho) + f_{jk}D_i(\rho)  + f_{ik}D_j(\rho) \right)\,,
$$
where $D_h$ are the total derivatives, see also  \eqref{eqn:total.derivatives}.
The Fubini-Pick invariant is the function defined as
\begin{equation}\label{eqPickContr}
    G^{i_1i_2}G^{j_1j_2}G^{h_1h_2}C_{i_1 j_1 h_1}C_{i_2 j_2 h_2}
\end{equation}
which, in the case $n=2$ and up to a non-zero factor, is equal to the right-hand side term of \eqref{eqFormulaSospirataXXX}.
%
%
Then the 
$\Aff(3)$--invariant PDE is $\E:=\{f=0\}$, with $f$ given by \eqref{eqFormulaSospirataXXX}, see   \cite{3rdOderAffPDE_preprint} for more details.
\par
Another approach, based on the study of the singularities of the group action, that has been used in \cite{MR2406036}, lead to the very same equation \eqref{eqFormulaSospirataXXX}.\par
We stress that the equation $\E$ projects onto the whole of $J^2$, because \eqref{eqFormulaSospirataXXX} is defined on the whole  $J^3$: however, if we take the intersections
\begin{equation*}
    \E\cap\check{J}^2_+\, ,\quad \E\cap\check{J}^2_-\, ,\quad 
\end{equation*}
we obtain precisely the two equations, say,  $\E_{\Sigma_+}$ and $\E_{\Sigma_-}$, that come from Theorem  \ref{corCasoAff}: they correspond to the $\CO(2)$--invariant subset  $\Sigma_+:=\{0\}$ and to the $\CO(1,1)$--invariant subset  $\Sigma_-$ made of two invariant lines, respectively. In other words,
\begin{equation*}
    \E=\overline{\E_{\Sigma_+}\cup\E_{\Sigma_-}}\, ,
\end{equation*}
whence the adjective ``unique''. 

\end{document}